\documentclass[12pt]{amsart}
\usepackage[all]{xy}
\usepackage{amssymb}
\usepackage{amsthm}
\usepackage{hyperref}
\hypersetup{colorlinks=true,linkcolor=blue,citecolor=magenta}
\usepackage{amsmath}
\usepackage{amscd,enumitem}
\usepackage{verbatim}
\usepackage{eurosym}
\usepackage{float}
\usepackage{color}
\usepackage{dcolumn}
\usepackage[mathscr]{eucal}
\usepackage[all]{xy}
\usepackage{hyperref}
\usepackage{bbm}
\usepackage[textheight=8.8in, textwidth=6.8in]{geometry}
\newtheorem*{thm*}{Theorem}
\newtheorem*{conj*}{Conjecture}

\newtheorem*{remark}{Remark}

\newtheorem{theorem}{Theorem}[section]

\newtheorem{proposition}[theorem]{Proposition}

\newtheorem*{example}{Example}

\newcommand{\CL}{\mathrm{CL}}
\newcommand{\Z}{\mathbb{Z}}
\newcommand{\Q}{\mathbb{Q}}
\newcommand{\R}{\mathbb{R}}
\newcommand{\SL}{\operatorname{SL}}

\newcommand{\tor}{\mathrm{tor}}
\newcommand{\Vol}{\operatorname{Vol}}
\numberwithin{equation}{section}

\begin{document}
\title{Elliptic curves and lower bounds for class numbers}
\author{Michael Griffin and Ken Ono}
\address{Department of Mathematics, University of Virginia, Charlottesville, VA 22904}
\email{ken.ono691@virginia.edu}
\address{Department of Mathematics, 275 TMCB, Brigham Young University, Provo, UT 84602}
\email{mjgriffin@math.byu.edu}

\begin{abstract} Ideal class pairings map the rational points of rank $r\geq 1$ elliptic curves $E/\Q$ 
to the ideal class groups $\CL(-D)$ of certain  imaginary quadratic fields.
These pairings imply that
 $$h(-D) \geq \frac{1}{2}(c(E)-\varepsilon)(\log D)^{\frac{r}{2}}
$$
for sufficiently large discriminants $-D$ in certain families, where $c(E)$ is a natural constant.
These bounds are effective, and they offer improvements to known  lower bounds for many discriminants.
\end{abstract}


\maketitle
\section{Introduction and statement of results}\label{Intro}

Estimating class numbers $h(-D)$  of imaginary quadratic fields
 $\Q(\sqrt{-D}),$ which also count equivalence classes of
  integral  positive definite binary quadratic forms of fundamental discriminant $-D$, is one of the oldest problems in
number theory. Gauss conjectured that $h(-D)\rightarrow +\infty$ as $D\rightarrow \infty$. Heilbronn \cite{Heilbronn} confirmed this  in the 1930s, and Siegel \cite{Siegel} shortly thereafter obtained a nearly definitive solution. For $\varepsilon>0,$ he proved that there are constants $c_1(\varepsilon),
c_2(\varepsilon)>0$ for which
$$
c_1(\varepsilon) D^{\frac{1}{2}-\varepsilon} \leq h(-D) \leq c_2(\varepsilon)D^{\frac{1}{2}+\varepsilon}.
$$
Siegel's lower bound is inexplicit;
there is no known formula for $c_1(\varepsilon)$. 
Therefore, his work could not even determine the class number 1 fundamental discriminants.  
Baker, Heegner, and Stark  \cite{Baker,  Heegner, Stark}  later famously determined this list:  $-D\in \{-3,-4,-7,-8,-11,-19,-43,-67,-163\}.$

Thanks to work of Goldfeld, Gross and Zagier, modified and  ingeniously optimized by Watkins, such lists are now known \cite{Watkins} for
  $h(-D)\leq 100$.
 The deep theorem
of Goldfeld \cite{Goldfeld1}, published in 1976, offered effective class number lower bounds assuming the existence of an elliptic curve $E/\Q$ with analytic rank $r\geq 3$. Groundbreaking work by Gross and Zagier \cite{Goldfeld2,GrossZagier} on the Birch and Swinnerton-Dyer Conjecture  confirmed the existence of such curves ten years later, resulting in the effective lower bound
 \cite{Oesterle}
\begin{equation}\label{GGZ}
h(-D) >
\frac{1}{7000}\left(\log D\right) \prod_{\substack{p\mid D\ {\text {\rm prime}}\\ p\neq D}} \left(
1-\frac{[2\sqrt{p}]}{p+1}\right).
\end{equation}

Here we obtain effective lower bounds for certain discriminants using elliptic curves in a completely different way.
Although we do not improve on (\ref{GGZ}) for all $-D$, the point of this note is to highlight and make use of an interesting interrelationship between class groups and elliptic curves
that often leads to improved class number lower bounds.
We employ
{\it  ideal class pairings}, maps of the form
 $$
 E(\Q)\times E_{-D}(\Q)\rightarrow \CL(-D),
 $$
 where $E_{-D}$ is the $-D$-quadratic twist of  $E$. Such maps
 were previously considered by Buell, Call, and Soleng \cite{Buell, BuellCall, Soleng}.
 The idea is quite natural. 
  Throughout, suppose that $E/\Q$ is given by
\begin{equation}\label{EModel}
E: \ \ y^2 = x^3 +a_4x +a_6,
\end{equation}
where $a_4, a_6\in \Z$, with $j$-invariant $j(E)$ and discriminant $\Delta(E)$, and
suppose that $E(\Q)$ has rank $r=r(E)\geq 1$.
If $-D<0$ is a fundamental discriminant, then let $E_{-D}/\Q$ be its quadratic twist\footnote{For reasons which will become apparent later, we choose this nonstandard normalization.}
\begin{equation}\label{twist}
E_{-D}: \  -D\cdot \left (\frac{y}{2}\right )^2=\, x^3 +a_4x+a_6.
\end{equation}
Furthermore, suppose that {\color{black}
$Q_D=({u},{v})\in E_{-D}(\Q)$ is an integer point,} where
 $v\neq 0$, with $v$ even if $-D$ is odd.\footnote{
 Goldfeld conjectures
 \cite{GoldfeldConjecture}
  that  asymptotically half of the $E_{-D}$ have rank 1, and so such points are plentiful. The integrality of $Q_D$
  is easily satisfied by changing models of $E$ and $E_{-D}$ by clearing denominators if necessary.}
Theorem~\ref{ThmQF} gives an explicit construction of the pairing. Therefore, the number of $\SL_2(\Z)$-inequivalent forms obtained by
 pairing  points in $E(\Q)$ with $Q_D$ gives a lower bound for $h(-D)$.

 We derive lower bounds  in terms of $\Omega_r:=\pi^{\frac{r}{2}}/\Gamma\left (\frac{r}{2}+1\right)$, the volume of the $\R^r$-unit ball, the regulator $R_{\Q}(E)$,  the diameter $d(E)$ (see (\ref{diameter})), and the torsion subgroup $E_{\tor}(\Q).$ We define
\begin{equation}
c(E):=\frac{|E_{\tor}(\Q)|}{\sqrt{R_{\Q}(E)}}\cdot \Omega_r,
\end{equation}
and, in terms of the usual logarithmic heights  (see Section~\ref{Heights}) of $j(E)$ and $\Delta(E),$ we define
\begin{equation}
\delta(E):={\color{black} \frac{1}{8}h_{W}(j(E))+\frac{1}{12}h_{W}(\Delta(E))+\frac{5}{3}}.
\end{equation}
Finally, to facilitate the comparison with $\log(D)$, we define
\begin{equation}\label{Tconstant}
T_E(D,Q_D):=\tfrac14\log\left(\frac{D}{(1+|u|)^2} \right)-\delta(E).
\end{equation}

\begin{theorem}\label{General}
Assuming the hypotheses above, if
{\color{black}$(1+|u|)^2\exp(4\delta(E)+d(E))<D\leq \frac{(1+|u|)^2u^2}{v^4}$}, then
$$
h(-D)\geq \frac{c(E)}{2}\cdot \left(T_E(D,Q_D)^{\frac{r}{2}}-r\sqrt{d(E)}\cdot  T_E(D,Q_D)^{\frac{r-1}{2}}\right).
$$
\end{theorem}

Although the hypotheses for Theorem~\ref{General} are satisfied by many $E/\Q$ for each $-D$,
we seek choices
that  improve (\ref{GGZ}).  In view of (\ref{Tconstant}), we require choices where the
height of $Q_D$ is small and where $c(E)$ is not too small. 
For each $E$, we offer a natural family of discriminants, those of the form
$-D_E(t):=-4(t^3+a_4t-a_6),$ with $t\in \Z$. In these cases, we choose $Q_{-D_E(t)}:=(-t,1)$.

\begin{theorem}\label{MainThm}
If $\varepsilon>0,$ then there is an effectively computable constant $N(E,\varepsilon)<0$ 
such that for negative fundamental discriminant of the form $-D_E(t)$, where $t\in \Z$ and
$-D_E(t)<N(E,\varepsilon),$ 
we have
$$
h(-D_E(t)) \geq \frac{1}{2} \left (\frac{c(E)}{\sqrt{12^r}}-\varepsilon\right)\cdot \log(D_E(t))^{\frac{r}{2}}.
$$
\end{theorem}

\begin{remark}
Theorems~\ref{General} and~\ref{MainThm}  are stated under the assumption that the discriminants are fundamental for reasons of aesthetics. There is a straightforward
modification that offers lower bounds for the class number of the corresponding imaginary quadratic field when $-D$ (resp. $-D_E(t)$) is not fundamental.  We note that Theorem~\ref{ThmQF}, which is the source of the lower bounds, holds for all discriminants. Namely, the proof gives lower bounds for
the Hurwitz-Kronecker class number $H(-D)$ (for example, see p. 273 \cite{C}), the class number of discriminant $-D$ quadratic forms, which counts each class $C$ with multiplicity $1/{\text {\rm Aut}}(C).$ Thankfully, Hurwitz class numbers satisfy a particularly nice multiplicative formula relative to class numbers with fundamental discriminant.  If $-D=-D_0 f^2,$ where $-D_0$ is a negative fundamental discriminant, then
\begin{equation}\label{Hecke}
H(-D)=\frac{h(-D_0)}{\omega(-D_0)}\cdot \sum_{d\mid f}\mu(d)\chi_{-D_0}(d)\sigma_1(f/d),
\end{equation}
where $\mu(\cdot)$ is the M\"obius function, $\omega(-D_0)$ is half the number of units in $\Q(\sqrt{-D_0}),$ $\chi_{-D_0}(\cdot)$ is the corresponding Kronecker character, and $\sigma_1(n)$ is the sum of positive divisors of $n$. As a result, a lower bound for $H(-D)$ (resp.  $H(-D_E(t))$) leads to a lower bound for the class
number of the corresponding imaginary quadratic field.
\end{remark}

\begin{remark} 
A classical theorem of Hooley (see Ch. IV of \cite{Hooley}) gives asymptotic formulas for the number of square-free values
of irreducible cubic polynomials $f(t)\in \Z[t]$. Namely, it is generally the case (i.e. barring trivial obstructions arising from congruence conditions) that a positive proportion of the values of $f$ at integer arguments are square-free.
Using this fact we can quantify the frequency with which
Theorem~\ref{MainThm}  improves on (\ref{GGZ}) for large $-D_E(t)$ (i.e. $t\rightarrow +\infty$) when $r(E)\geq 3$. 
A famous example of Elkies \cite{Elkies} has $r(E)\geq 28,$ and so we obtain
the effective lower bound 
$$
h(-D)\gg_{\varepsilon} (\log D)^{14-\varepsilon}
$$
which holds for $\gg_{\varepsilon} X^{\frac{1}{3}}$ many explicit  fundamental discriminants $-X<-D<0$.
\end{remark}

\begin{example} For 
$
E: \  y^2 = x^3-16x+1,
$
we have\footnote{These calculations were performed using
\texttt{SageMath}.} $|E_{\tor}(\Q)|=1$, $r(E)=3$, and $R_E(\Q)\sim 0.930\dots$.
Therefore,   for large fundamental discriminants of the form
$-D_E(t)=-4(t^3-16t-1)$,  we have
$$
h(-D_E(t)) > \frac{1}{20}\cdot (\log(D_E(t))^{\frac{3}{2}}.
$$
\end{example}

We give infinite families of $E/\Q$ using the discriminant $\Delta_{a,b}:=-16(27b^{4}-4a^6)$ curves
\begin{equation}\label{E1mn}
E_{a,b} \ : \ y^2=x^3-a^2x+b^2.
\end{equation}
For  integers  $t$, we let
$D_{a,b}(t):=4(t^3-a^2t-b^2).$
For positive integers $a, b$, we let
\begin{equation}
c^{(2)}_{a,b}:=\frac{\Omega_2}{12\cdot \widehat{h}(P^{(2)}_{\max})} \ \ \  \ \ {\text {\rm and}}\ \ \ \ \ 
c^{(3)}_{a,b^3}:=\frac{\Omega_3}{24\sqrt{3}\cdot \widehat{h}(P^{(3)}_{\max})^{\frac{3}{2}}},
\end{equation}
where $P^{(2)}_{\max}\in \{(0,b), (-a,b)\}\subset E_{a,b}(\Q)$ and
$P^{(3)}_{\max}\in \{ (0,b^3), (-a,b^3), (-b^2,ab)\}\subset E_{a,b^3}(\Q)$
are chosen to have the largest canonical height.

\begin{theorem}\label{Thm2}
If $a$ and $b$ are positive integers,  then the following are true:
\begin{enumerate}
\item If $a\gg_b 1$ (resp. $b\gg_a 1$), then $r(E_{a,b}(\Q))\geq 2$. Moreover, if $\varepsilon>0,$ then for sufficiently large
fundamental discriminants $-D_{a,b}(t)<0$ in absolute value we have
$$
h(-D_{a,b}(t)) \geq  (c^{(2)}_{a,b}-\varepsilon)\cdot \log(D_{a,b}(t)).
$$

\item If $a\gg_b 1$ (resp. $b\gg_a 1$), then $r(E_{a,b^3}(\Q))\geq 3$. Moreover, if $\varepsilon>0,$ then for sufficiently large
fundamental discriminants $-D_{a,b^3}(t)<0$ in absolute value we have
$$
h(-D_{a,b^3}(t)) \geq  (c^{(3)}_{a,b^3}-\varepsilon)\cdot \log(D_{a,b^3}(t))^{\frac{3}{2}}.
$$
\end{enumerate}
\end{theorem}

\noindent
{\bf Three Remarks.}

\noindent
(1) Theorem~\ref{Thm2} is effective. One can make explicit\footnote{For example, Fujita and Nara \cite{FujitaNara} show for $b=1$ that $a\geq 4$ suffices in Theorem~\ref{Thm2} (2).} $b\gg_a 1$ and $a\gg_b 1.$ Moreover, we note that $-D_{n+1,n}(-1)=-8n$ covering all
$-D\equiv 0\pmod{8}$. Using (\ref{Hecke}) and the remark after Theorem~\ref{MainThm}, we find that these curves cover all negative discriminants.
\smallskip

\noindent
(2) Theorem~\ref{Thm2} (1)
often improves on (\ref{GGZ}) (e.g.
for large $t\in \Z^{+}$ when $a$ and $b$ are small, or when $-D_{a,b}(t)$ is suitably composite).

\smallskip
\noindent
(3) Theorem~\ref{Thm2} (2) often provides a $\log D$ power improvement to (\ref{GGZ}).  Florian Luca has noted, for each $0<c<1/2$, that the effective lower bound
$$
h(-D) \gg_c \log(D)^{\frac{3}{2}-c}
$$
holds for $\gg X^{\frac{1}{3}}\cdot \exp\left(\frac{4}{3}\log(X)^{\frac{c}{3}}\right)$ 
many explicit $-X<-D< 0$.  The idea is that integers of the form $N=t^3-a^2t-b^2$  have unique representations
with $t\in [X/2,X],$ $a\in [y/2,y]$, and
$b\in [z/2,z],$ where $y=o(X^{\frac{1}{3}})$ and $z=o(X^{\frac{1}{6}})$, and one then lets $y=z=\exp\left(\log(X)^{\frac{c}{3}}\right)$ and counts cubes $b$.

\medskip

This note is organized as follows. In Section~\ref{NutsAndBolts} we prove Theorem~\ref{ThmQF},  a result which provides the ideal class pairings, and
determines conditions guaranteeing $\SL_2(\Z)$-inequivalence. Using this result, the proof of Theorem~\ref{MainThm} is
reduced to effectively counting rational points with bounded height, which we address
in Section~\ref{Heights}.  In Section~\ref{TheProof} we state and prove Theorem~\ref{Thm1}, a result which
implies Theorem~\ref{MainThm}. Theorem~\ref{General} follows the proof of Theorem~\ref{Thm1} {\it mutatis mutandis}.
Finally, in Section~\ref{family} we prove Theorem~\ref{Thm2}.

\section*{Acknowledgements}  The second author thanks the NSF (DMS-1601306) and
the Thomas Jefferson fund at the U. Virginia. The authors thank the referee, N. Elkies, D. Goldfeld, B. Gross, F. Luca, K. Soundararajan, D. Sutherland and J. Thorner for  useful comments concerning this paper.

\section{Elliptic curves Ideal class pairings}\label{NutsAndBolts}

Works by Buell, Call, and Soleng \cite{Buell, BuellCall, Soleng} offered elliptic curve ideal class pairings, which
produce discriminant $-D$ integral positive definite binary quadratic forms from points on
$E(\Q)$ and $E_{-D}(\Q)$.
We offer a generalization and minor correction of Theorem 4.1 of \cite{Soleng}.\footnote{This corrects sign errors in the discriminants in Theorem 4.1 of \cite{Soleng}, and also ensures the resulting quadratic forms are integral when $C\neq 1$. Moreover, this theorem allows for both even and odd discriminants.}

Assume the notation from Section~\ref{Intro}.
Let $P=(\tfrac{A}{C^2},\tfrac{B}{C^3})\in E(\Q),$ with $A, B, C\in \Z$, and $Q=(\tfrac{u}{w^2},\tfrac{v}{w^3})\in E_{-D}(\Q),$
with $u,v,w\in \Z,$ not necessarily in lowest terms\footnote{Thanks to (\ref{twist}), every $Q$ has such a representation
where $\gcd(u,w^2)$ and $\gcd(v,w^3)$ divide $D$.}. Moreover, suppose that
 $v\neq 0$, with $v$ even if $-D$ is odd.
If we let $\alpha:=|Aw^2-u C^2|$ and  $G:=\gcd(\alpha, C^6v^2),$ then we shall show that there are integers $\ell$
for which $F_{P,Q}(X,Y)$ defined below is a discriminant $-D$ positive definite integral binary quadratic form. 
\begin{equation}
{\color{black}
F_{P,Q}(X,Y)=\frac{\alpha}{G} \cdot X^2+\frac{2w^3 B+\ell  \cdot \tfrac{\alpha}{G}}{C^3v}\cdot XY +\frac{\left({ 2w^3B+
\ell \cdot \tfrac{\alpha}{G}}\right)^2+C^6 v^2{D}}{4C^6v^2\cdot \frac{\alpha}{G}}\cdot Y^2
}
\end{equation}

\begin{theorem}\label{ThmQF}
Assuming the notation and hypotheses above, $F_{P,Q}(X,Y)$ is well defined (e.g. there is such an $\ell$)  {\color{black} in $\CL(-D)$}.
Moreover, if $(P_1,Q_1)$ and $(P_2, Q_2)$ are two such pairs for which $F_{P_1,Q_1}(X,Y)$ and $F_{P_2,Q_2}(X,Y)$ are $\SL_2(\Z)$-equivalent, then $\frac{\alpha_1}{G_1}=\frac{\alpha_2}{G_2}$ or $\frac{\alpha_1\alpha_2}{G_1G_2}>D/4$.
\end{theorem}

\begin{example} For $E: y^2=x^3-4x+9$, we have points $P_1:=(0,3)$ and $P_2:=(-2,3)$.
Using $Q:=(-3,1)\in E_{-24}(\Q)$ and $\ell=2$, we obtain the inequivalent discriminant $-24$ forms
$F_{P_1,Q}(X,Y)=3X^2+12XY+14Y^2$ and $F_{P_2,Q}(X,Y)=X^2+8XY+22Y^2.$ It turns out $h(-24)=2$.

\end{example}

\begin{proof}
A calculation
shows that $F_{P,Q}(X,Y)$  has discriminant $-D$. 
We now show that there are integers $\ell$ for which  $F_{P,Q}(X,Y)$ is integral,
and that all such choices
preserve $\SL_2(\Z)$-equivalence.
To this end, let 
$f(x):= x^3+a_4x+a_6,$
so that
$B^2=C^6 f\left(\frac{A}{C^2}\right)$ and $ -v^2D=4w^6f\left(\frac{u}{w^2}\right).$
Note that
$\alpha =\left |w^2C^2 \left(\frac{A}{C^2}-\frac{u}{w^2}\right)\right |,$
 which divides 
 \begin{equation}\label{alphadivs}
 w^6B^2+C^6v^2\tfrac{D}{4} \,=\, w^6C^6\left( f\left(\frac{A}{C^2}\right)-f\left(\frac{u}{w^2}\right)\right).
 \end{equation}
\noindent Since $G=\gcd(\alpha,C^6v^2)$, we have that $G \mid 4w^6B^2$. Let $H:=\gcd(2w^3B,C^3v).$
  Then $G\mid H^2$, and so $C^3v/H$, which divides $C^6v^2/G$, is relatively prime to $\alpha/G$. Choose $k\in \Z$ so that
  \[
  \frac{\alpha k}{G}\equiv -\frac{2w^3B}{H}-\frac{C^3vD}{H} \pmod{\tfrac{2C^3v}{H}}.
  \]
If $\alpha/G$ is odd, $k$ can be found by inverting $\alpha/G\pmod{\tfrac{2C^3v}{H}}.$ If $\alpha/G$ is even, then $-C^3vD\equiv 2w^3B \pmod{2}$, and so $k$ may be found by inverting $\alpha/2G\pmod{\tfrac{C^3v}{H}}.$ We  take $\ell \equiv Hk \pmod{2C^3v}$ or $\pmod{C^3v}$ depending on whether $k$ is defined $\pmod{\tfrac{2C^3v}{H}}$ or $\pmod{\tfrac{C^3v}{H}}$ respectively.
 The conditions on $\ell$ imply that the coefficient of $XY$ in $F_{P,Q}(X,Y)$ has the same parity as $-D$.  The numerator of the $Y^2$ term,
$\left({2w^3B+\ell \alpha/G}\right)^2+C^6 v^2D$, is divisible by $4C^6v^2.$
By (\ref{alphadivs}), it is also divisible by $4\alpha$. Therefore, it is divisible by  $4C^6v^2\alpha/ G$, and so $F_{P,Q}(X,Y)$  is integral.

We now determine the inequivalence of $F_{P_1,Q_1}(X,Y)$ and $F_{P_2,Q_2}(X,Y)$.
For $i=1$ and $2$ we let $A_i, B_i, C_i, \alpha_i,$ and $G_i$ be the corresponding quantities for these two pairs of points.
 Note that
\begin{align*}
F_{P_1,Q_1}(X,Y)=\frac{G_1}{\alpha_1}\left[\left(\frac{\alpha_1}{G_1} X+\frac{2w^3 B_1+\ell\cdot \frac{\alpha}{G}}{2C_1^3v} Y\right)^2 +\frac{D}{4} Y^2\right].
\end{align*}
Since $\ell$ was chosen so that $\ell\cdot \frac{\alpha}{G}$ is defined modulo $2C^3v,$ its choice does not affect $\SL_2(\Z)$-equivalence.
If $\left(\begin{smallmatrix} a&b\\c&d\end{smallmatrix}\right)\in \SL_2(\Z)$ and
$F_{P_2,Q_2}(X,Y) = F_{P_1,Q_1}(aX+bY,cX+dY),$  then the leading terms satisfy
\[
\frac{\alpha_2}{G_2} =  \frac{G_1}{\alpha_1}\left[\left(\frac{\alpha_1 a}{G_1}+\frac{2w^3 B_1+\ell\cdot \frac{\alpha}{G}}{2C_1^3v} c\right)^2 +\frac{D}{4} c^2\right].
  \]
If $c=0$, then $a^2=1$, and the equation reduces to $\frac{\alpha_2}{G_2}=\frac{\alpha_1}{G_1}$.
If $c\neq 0$, both terms inside the square brackets are positive, and together are at least $D/4$, so $\frac{\alpha_2}{G_2}\geq \frac{G_1}{\alpha_1}D/4$.
\end{proof}

\section{Heights on elliptic curves}\label{Heights}

To deduce Theorem~\ref{MainThm} from Theorem~\ref{ThmQF}, we use estimates for the number of rational points on elliptic
curves with bounded height. Here we recall the facts we require.
Each rational point $P\in E(\Q)$ has the form $P=(\frac{A}{C^2},\frac{B}{C^3})$, with $A,B,C$ integers such that $\gcd(A,C)=\gcd(B,C)=1$. The naive height of $P$ is
$H(P)=H(x):=\max(|A|, |C^2|).$ 
The logarithmic height (or Weil height) is 
$h_W(P)= h_W(x):=\log H(P),$
and the canonical height is given by 
\begin{equation}
\widehat h(P)= \tfrac{1}{2}\lim_{n\to \infty}\frac{h_W(nP)}{n^2}.
\end{equation}

Logarithmic and canonical heights are generally close. 
A theorem of  Silverman~\cite{Silverman} bounds the differences between these heights in terms
of the logarithmic heights of $j(E)$ and $\Delta(E)$. 

\begin{theorem}[Theorem 1.1 of \cite{Silverman}]\label{Silverman_bounds}
If $P\in E(\Q)$, then
\[
-\tfrac{1}{8}h_W(j(E))-\tfrac{1}{12}h_W(\Delta(E))-0.973\leq \widehat{h}(P)-\frac{1}{2}h_W(P)\leq \tfrac{1}{12}h_W(j(E))+\tfrac{1}{12}h_W(\Delta(E))+1.07.
\]
\end{theorem}

Asymptotics for the number of rational points on an elliptic curve with bounded height are well known
(for example, see \cite[Prop 4.18]{Knapp}). If
 $E(\Q)$ has rank $r\geq 1$ and
$\Omega_r=\pi^{\frac{r}{2}}/\Gamma\left (\frac{r}{2}+1\right)$,  then in terms of the
 regulator $R_{\Q}(E)$ and $|E_{\tor}(\Q)|$, we have
\begin{equation}
\# \{P\in E(\Q) \mid \widehat h(P) \leq T\}
\sim \frac{|E_{\tor}(\Q)|}{\sqrt{R_{\Q}(E)}} \cdot \Omega_r T^{\frac{r}{2}}=c(E) T^{\frac{r}{2}}.
\end{equation}
Using an argument of Landau which estimates the number of lattice points in $r$-dimensional spheres (for example, see \cite{Landau}), one can show that the error term in the asymptotic is  $\text{O} (T^{\frac{r}{2}-1+\frac{1}{r+1}}).$

To prove Theorem~\ref{MainThm}, 
we require effective lower bounds for the number of points with bounded height.
To this end, if $\{P_1,\dots,P_r\}$ is a basis of $E(\Q)/E_{\tor}(\Q)$, then its diameter is
\begin{equation}\label{diameter}
d(E)=\max_{\delta_i\in \{\pm1,0\}} 2\widehat h\left(\sum_{i=1}^r\delta_iP_i\right ).
\end{equation}
It is the largest square-distance between any two vertices of the parallelopiped in $\R^r$ constructed from vectors $\textbf{v}_1,\dots \textbf{v}_r$ which have $\textbf{v}_i\cdot \textbf{v}_j=\langle P_i,P_j\rangle:=
{\color{black}\frac{1}{2}\left(\widehat{h}(P_i+P_j)-\widehat{h}(P_i)-\widehat{h}(P_j)\right).}$

\begin{proposition}\label{PropRawBounds} Assume the notation and hypotheses above.
If $d=d(E)$ is the diameter of any basis of $E(\Q)/E_{\tor}(\Q),$ then for $T>d(E)/4$ we have 
\[
\# \{P\in E(\Q) \mid \widehat h(P) \leq T\}\geq c(E) \left(T^{\frac{r}{2}}-r
\sqrt{d}\cdot T^{\frac{r-1}{2}} \right).
\]

\end{proposition}

\begin{proof}
Let $\mathcal B=\{P_1,\dots,P_r\}$ be any basis for $E(\Q)$.
We must count points on the lattice $\Lambda\in \R^r,$ generated by $v_1,v_2,\dots, v_r$ for which $v_i\cdot v_j=\langle P_i,P_j\rangle.$ The number of points in the subgroup of $E(\Q)$ generated by $\mathcal B$ with canonical height bounded by $T$ is  the number of points in $\Lambda \cap B(T^{\frac{1}{2}})$, where $B(R)$ is the closed ball in $\R^r$ centered at the origin of radius $R$. 

For each point $\lambda \in \Lambda$, let $P_\lambda$ be the half-open parallelepiped given by 
\[\mathcal P_\lambda=\left \{\lambda +\sum_{i=1}^r x_i \mathbf{v}_i ~\mid~ x_i\in [0,1)\right\}.\]
If $P_\lambda$ intersects $B(T^{\frac{1}{2}}-d^{\frac{1}{2}}),$ then $\lambda\in B(T^{\frac{1}{2}}).$  Therefore, we have
\begin{eqnarray*}
\# \left(\Lambda \cap B(T^{\frac{1}{2}}) \right)&\geq \frac{\Vol \left(B(T^{\frac{1}{2}}-d^{\frac{1}{2}})\right)}{\Vol (\mathcal P_\lambda)}
= \frac{\Omega_r}{\Vol (\mathcal P_\lambda)}\cdot \left(T^{\frac{1}{2}}-d^{\frac{1}{2}} \right)^r
\\
&\geq \frac{\Omega_r}{\Vol (\mathcal P_\lambda)}\cdot \left(T^{\frac{r}{2}}-r\sqrt{d}\cdot T^{\frac{r-1}{2}} \right).
\end{eqnarray*}
In the last inequality we used the binomial expansion and the fact that $T\geq d/4$.
Since we have $R_{\Q}(E):=|\det(\langle P_i,P_j\rangle)_{1\leq i,j\leq r}|,$ it follows that
$\Vol (\mathcal P_\lambda)=\sqrt{R_{E}(\Q)}.$ To complete the proof, we note that
torsion points have  height zero, and so we may multiply the last estimate by $|E_{\tor}(\Q)|.$ 
\end{proof}

These same arguments can be used to give lower bounds for the number of points of bounded height generated
from any linearly independent points in $E(\Q)$.

\begin{proposition}\label{PropBoundsApprx} Assume the notation and hypotheses above.
Suppose $G$ is a subgroup of $E_{\tor}(\Q)$, and that $\mathcal B=\{P_1,\dots,P_m\}$ is a set of linearly independent points in $E(\Q)$ listed in ascending order by height. If $T>d(\mathcal{B})/4,$ then
\[
\# \{P\in E(\Q) \mid \widehat h(P) \leq T\}\geq \frac{|G|}{ \sqrt{\widehat h(P_m)^m}}\cdot \Omega_m \left(T^{\frac{m}{2}}-m^2\sqrt{2\widehat h(P_m)}T^{\frac{m-1}{2}} \right).
\]
\end{proposition}
\begin{proof} The proof of Proposition~\ref{PropRawBounds} works with two modifications.
Note that
${d(\mathcal B)} \leq {2m^2\widehat h (P_m)}$, and that the volume of the parallelopiped for $\mathcal B$ satisfies
$\Vol (\mathcal B)\leq \prod_{i=1}^m \widehat h(P_i)^{1/2}\leq \widehat h(P_r)^{\frac{m}{2}}.$
\end{proof}

\section{Proof of Theorems~\ref{General} and ~\ref{MainThm}}\label{TheProof}
Theorems~\ref{General} and \ref{MainThm} are proven in the same way. For simplicity,
we first consider Theorem~\ref{MainThm}, which pertains to fundamental discriminants $-D_E(t)=-4(t^3+a_4t-a_6)$, and where
we have chosen to pair the points in $E(\Q)$ with $Q_t:=(-t,1)\in E_{-D_E(t)}(\Q)$. We obtain
the precise Theorem~\ref{Thm1}, which in turn implies Theorem~\ref{MainThm}.

We show  that points $P\in E(\Q)$ with canonical height
$\widehat{h}(P)\leq 
T_E(t)$, where
\begin{equation}\label{Tn}
T_E(t):=\frac{1}{4}\log\left(\frac{D_E(t)}{(t+1)^2}\right)-\delta(E),
\end{equation}
map to inequivalent forms $F_{P,Q_t}(X,Y)\in \CL(-D)$.

\begin{theorem}\label{Thm1} Assume the hypotheses above.
If $T_E(t)\geq d(E)/4$ and 
 $-D_E(t)$ is a negative fundamental discriminant for which
$(t+1)^2\exp(4\delta(E)+d(E))\leq D_E(t)\leq t^2(t+1)^2,
$
then
$$
h(-D_E(t))\geq \frac{c(E)}{2}\left(
T_E(t)^{\frac{r}{2}}-r \sqrt{d(E)}
 T_E(t)^{\frac{r-1}{2}}\right).
$$
\end{theorem}

\begin{remark}
Since $D_E(t)$ is cubic, the conclusion holds for all but finitely many  $-D_E(t)$. Moreover,
the proof works for any $m$ independent points in $E(\Q)$ thanks to Proposition~\ref{PropBoundsApprx}.
\end{remark}

\begin{proof}[Deduction of Theorem~\ref{MainThm} from Theorem~\ref{Thm1}]
Due to the $(t+1)^2$ in (\ref{Tn}), we have
$T_E(t)\sim  \log(D_E(t))/12$, and the result follows.
\end{proof}

\begin{proof}[Proof of Theorem~\ref{Thm1}]
We suppose that $t\in \Z^{+}$ satisfies
$(t+1)^2\exp(4\delta(E)+d(E))\leq D_E(t)\leq t^2(t+1)^2.$
Proposition \ref{PropRawBounds} implies that
\begin{equation}\label{MainCount}
\# \{P\in E(\Q) \mid \widehat h(P) \leq T_E(t)\}\geq \frac{|E_{\tor}(\Q)|}{\sqrt{R_{\Q}(E)}}\cdot \Omega_r \left(T_E(t)^{\frac{r}{2}}-r\sqrt{d(E)}T_E(t)^{\frac{r-1}{2}} \right).
\end{equation}
We show that these points map to inequivalent forms when paired with $Q_t=(-t,1)\in E_{-D_E(t)}(\Q)$.
 
Suppose that  $P_1=(\tfrac{A_1}{C_1^2},\tfrac{B_1}{C_1^3}), P_2=(\tfrac{A_2}{C_2^2},\tfrac{B_2}{C_2^3})\in E(\Q)$ satisfy $\widehat h(P_i) \leq T_E(t)$, and let $F_1:=F_{P_1,Q_t}(X,Y)$ and $F_2:=F_{P_2,Q_t}(X,Y)$. Since $\gcd(A_i,C_i)=1,$  their leading terms are $\frac{\alpha_i}{G_i}=\alpha_i=|A_i+tC_i^2|.$ 
Thanks to Theorem~\ref{Silverman_bounds}, we have  that 
\[h_W(P_i)\leq 2\left(\widehat h(P_i)+\tfrac18 h_W(j(E))+\tfrac{1}{12}h_W(\Delta(E))+0.973\right)\leq 2T_E(t)-\log(2)=\tfrac12 \log\left|\frac{D_E(t)}{4(t+1)^2}\right|.\]
We observe that  $\alpha_i\leq (t+1)H(P_i)$. By Theorem~\ref{Silverman_bounds}, we have
$H(P_i)=\exp(h_W(P_i))\leq \frac{\sqrt{D_E(t)}}{2(t+1)},$
 which gives $\alpha_i\leq\tfrac12\sqrt{D_E(t)}.$ Hence, we find that
$\alpha_1\alpha_2\leq\tfrac{1}{4}D_E(t),$
and so by Theorem~\ref{ThmQF}, $F_1$ and $F_2$ are inequivalent unless 
$\alpha_1=|A_1+tC_1^2|=|A_2+tC_2^2|=\alpha_2.$
However, by hypothesis $D_E(t)\leq t^2(t+1)^2,$ and so $|A_i|\leq H(P_i)\leq \frac{t}{2}.$ Since $C_i^2>0$, this means $\alpha_i=|A_i+tC_i^2|=A_i+tC_i^2.$ If $\alpha_1=\alpha_2$, then $A_1\equiv A_2\pmod{t},$ and by the bounds on $|A_i|$ we have that $A_1=A_2.$ This then implies that $C_1^2=C_2^2,$ and so $P_1=\pm P_2,$ which explains the further
factor of $1/2$ that appears in the lower bound.
This completes the proof.
\end{proof}

\begin{proof}[Proof of Theorem~\ref{General}]

We follow the proof of Theorem~\ref{Thm1}, noting instead that
$Q_D=(u,v)\in E_{-D}(\Q),$
with $u,v\in \Z,$ where
 $v\neq 0$, and $v$ is even if $-D$ is odd.
Arguing as before, we find that $F_1$ and $F_2$ are inequivalent unless 
$\frac{\alpha_1}{G_1}=\frac{|A_1-uC_1^2|}{G_1}=\frac{|A_2-uC_2^2|}{G_2}=\frac{\alpha_2}{G_2},$ and that
$|A_i|\leq H(P_i)\leq \frac{|u|}{2v^2}\leq \tfrac{1}{2}|u|$. Since $C_i^2>0$, this implies that $A_1-uC_1^2$ and $A_2-uC_2^2$ have the same signs. If $\tfrac{\alpha_1}{G_1}=\tfrac{\alpha_2}{G_2}$, then  
\[
A_1G_2- A_2G_1 = u(C_1^2G_2 - C_2^2G_1).
\]
The left hand side is divisible by, but not exceeding, $|u|$; and so it must be $0$. This implies 
$\frac{A_1}{G_1}=\frac{A_2}{G_2}$, and $\frac{C^2_1}{G_1}=\frac{C^2_2}{G_2}$, and so
$\frac{A_1}{C_1^2}=\frac{A^2}{C^2_2}.$ Hence, we have $P_1=\pm P_2,$ which explains the further
factor of $1/2$ that appears in the lower bound.
This completes the proof.

%
%

\end{proof}

\section{A nice family of elliptic curves}\label{family}

Theorem~\ref{Thm2} is a simple consequence of the following proposition.

\begin{proposition}\label{FamilyEC} If $a$ and $b$ are  positive integers, then the following are true:
\begin{enumerate}
\item If $a\gg_b 1$ (resp. $b\gg_a 1$), then $(0,b)$ and $(-a,b)$ are independent points in
$E_{a,b}(\Q)$.
\item If $a\gg_b 1$ (resp. $b\gg_a 1$), then $(0,b^3),$ $(-a,b^3),$ and $(-b^2,ab)$ are independent
points in $E_{a,b^3}(\Q)$.
\end{enumerate}
\end{proposition}
\begin{proof}
This follows easily from Silverman's specialization theorem for elliptic curves. Suppose that  $E_t/\Q(t)$ is an elliptic curve which is not isomorphic over $\Q(t)$ to an elliptic curve defined over $\Q$. For $w\in \Q,$ we let $\sigma_w$ be the specialization map ($t\rightarrow w$):
\[
\sigma_w\,:\, E_t(\Q_t) \to E_w(\Q).
\]
Generally, $E_w$ is an elliptic curve over $\Q$.
Silverman's theorem (see Th. C of ~\cite{SilvermanSpecialization})  states, for all but finitely many $w\in \Q$, that $\sigma_w$
is an injective homomorphism between elliptic curves.
The claims follows immediately by viewing $a$ and $b$ as indeterminates respectively.

\end{proof}

\begin{proof}[Proof of Theorem~\ref{Thm2}]
In view of Proposition~\ref{FamilyEC} and Proposition~\ref{PropBoundsApprx}, the proof follows  the proof of Theorem~\ref{Thm1}
{\it mutatis mutandis}.
\end{proof}

\end{document}